\documentclass[11pt,a4paper]{article}

\usepackage[T1]{fontenc}
\usepackage[utf8]{inputenc}
\usepackage{lmodern}
\usepackage{amsmath,amssymb,amsthm,mathtools}
\usepackage{booktabs,array}
\usepackage{microtype}
\usepackage[margin=1in]{geometry}
\usepackage{hyperref}
\usepackage{fancyhdr}
\usepackage{enumitem}

\hypersetup{hidelinks,pdfauthor={Sen Cao and Sanjit Singh Mehat},pdftitle={Albertson's Conjecture for Chromatic Numbers at Most 29}}
\newtheorem{theorem}{Theorem}[section]
\newtheorem{lemma}[theorem]{Lemma}
\newtheorem{proposition}[theorem]{Proposition}
\newtheorem{corollary}[theorem]{Corollary}
\theoremstyle{remark}
\newtheorem{remark}[theorem]{Remark}

\newcommand{\crn}{\operatorname{cr}}
\newcommand{\e}{\operatorname{e}}
\newcommand{\odd}{\operatorname{o}}

\newcommand{\join}{\vee}
\newcommand{\comp}[1]{\overline{#1}}

\title{\textbf{Albertson's Conjecture for Chromatic Numbers at Most 29}\\[0.4em]
\large Critical graphs, essential immersions, Tutte barriers, and crossing-number sampling}
\author{%
Sen Cao\\
\small Wuchang Shouyi College, Wuhan, China
\and
Sanjit Singh Mehat\\
\small Independent Researcher, United States
}
\date{}

\begin{document}
\maketitle
\thispagestyle{plain}

\begin{abstract}
Albertson's conjecture asserts that every finite simple graph $G$ with $\chi(G)\ge r$ satisfies $\crn(G)\ge \crn(K_r)$. Building on Cranston's verification for $r\le 24$ and his reduction of $r\in\{25,26\}$ to three residual orders, we eliminate those residual cases and then prove the cases $r=27,28,29$.

The first structural ingredient is a Kempe-chain construction: if a $k$-critical graph has a vertex of degree $k-1$, then it contains a branch-clean essential immersion of $K_k$. Essential immersions are crossing-monotone, so a critical counterexample must have minimum degree at least $k$. For $r=27$, this one-unit degree gain, Gallai's join structure, critical-graph edge bounds, and induced-subgraph averaging close every possible order. For $r=28$ and $r=29$, the remaining near-$2r$ orders are converted to dense complements. Stehl\'ik's coloring theorem makes the odd-order complements factor-critical; a clique-partition obstruction yields an anti-tight matching property; and Tutte barriers, Hall-type expansion, and deficit bookkeeping eliminate the final cases. At order $58$ for $r=29$, Rabern's coloring inequality handles the regular case, while the last degree-deficit-two case is reduced to two disjoint triangles and a finite barrier analysis.

\end{abstract}

\section{Introduction and main result}
For a finite simple graph $G$, let $\crn(G)$ denote its crossing number and $\chi(G)$ its chromatic number. Albertson's conjecture states
\begin{equation}
\chi(G)\ge r \quad\Longrightarrow\quad \crn(G)\ge \crn(K_r).
\end{equation}
The conjecture was proved for $r\le 12$ by Albertson, Cranston, and Fox~\cite{AlbertsonCranstonFox2009}, for $r\le 16$ by Bar\'at and T\'oth~\cite{BaratToth2010}, and for $r\le 18$ by Ackerman~\cite{Ackerman2019}. Recent immersion-based progress was obtained by Fox, Pach, and Suk~\cite{FoxPachSuk2025}. Cranston subsequently proved the conjecture for $r\le 24$ and reduced the cases $r\le 26$ to three residual parameter pairs~\cite{Cranston2025}.

While this manuscript was being finalized, Sadhu independently proved Albertson's conjecture for $r\le 26$ and showed that any $27$-critical counterexample must have order $53$ or $54$ with connected complement~\cite{Sadhu2026}. Our proofs of the cases $r=25,26$, as well as our proof of the $r=27$ case, had already been completed and circulated prior to the appearance of Sadhu's preprint. The present work further proves the cases $r=28,29$.

The purpose of this paper is to prove the following extension.

\begin{theorem}[Main theorem]\label{thm:main}
For every integer $1\le r\le 29$ and every finite simple graph $G$,
\[
\chi(G)\ge r \Longrightarrow \crn(G)\ge \crn(K_r).
\]
Equivalently, Albertson's conjecture holds for every $r\le 29$.
\end{theorem}

We use the standard two-circle drawing upper bound
\begin{equation}
\crn(K_r)\le Z(r):=\frac14
\left\lfloor\frac r2\right\rfloor
\left\lfloor\frac{r-1}{2}\right\rfloor
\left\lfloor\frac{r-2}{2}\right\rfloor
\left\lfloor\frac{r-3}{2}\right\rfloor.
\end{equation}
For the values used below,
\[
Z(25)=4356,\quad Z(26)=5148,\quad Z(27)=6084,\quad Z(28)=7098,\quad Z(29)=8281.
\]
Thus it suffices in each new case to force a lower bound at least $Z(r)$; the numerical comparisons below are strict.

A bibliographic point matters for the numerical argument. We use the published bound of B\"ungener--Kaufmann~\cite{BungenerKaufmann2024}, which states the unconditional inequality
\[
\crn(F)\ge 5\e(F)-\frac{203}{9}(|V(F)|-2).
\]
This is the bound supporting the constant used throughout the present paper.

\section{Preliminaries}
\subsection{Critical graphs, subdivisions, and complements}
A graph $G$ is $r$-critical if $\chi(G)=r$ and every proper subgraph has chromatic number less than $r$.

\begin{lemma}[Critical reduction]\label{lem:critical-reduction}
Every graph $G_0$ with $\chi(G_0)\ge r$ contains an $r$-critical subgraph.
\end{lemma}
\begin{proof}
Choose a subgraph $G\subseteq G_0$ minimal under inclusion subject to $\chi(G)\ge r$. If $\chi(G)\ge r+1$, then for every vertex $v$, $\chi(G-v)\ge \chi(G)-1\ge r$, contrary to minimality. Hence $\chi(G)=r$, and minimality implies that every proper subgraph has chromatic number at most $r-1$.
\end{proof}

\begin{lemma}[Minimum degree]\label{lem:min-degree}
If $G$ is $r$-critical, then $\delta(G)\ge r-1$.
\end{lemma}
\begin{proof}
For $v\in V(G)$, criticality gives an $(r-1)$-coloring of $G-v$. If $d(v)\le r-2$, at most $r-2$ colors occur on $N(v)$, so one color is available for $v$, a contradiction.
\end{proof}

Crossing number is monotone under taking subgraphs. We also use subdivision invariance.

\begin{lemma}[Subdivision invariance]\label{lem:subdivision}
If $H$ is a subdivision of a graph $F$, then $\crn(H)=\crn(F)$. Consequently, if $G$ contains a subdivision of $K_r$, then $\crn(G)\ge \crn(K_r)$.
\end{lemma}
\begin{proof}
A drawing of $F$ gives a drawing of $H$ with the same crossings by placing each subdivision vertex on its edge-arc. Conversely, suppressing degree-two subdivision vertices in a drawing of $H$, with arbitrarily small local perturbations if necessary, gives a drawing of $F$ with no additional crossings.
\end{proof}

If $A$ and $B$ are vertex-disjoint graphs, their join $A\join B$ is obtained by adding every edge between $V(A)$ and $V(B)$. Chromatic number is additive under joins. If a critical graph decomposes nontrivially as a join, then each join factor is critical with respect to its own chromatic number: otherwise replacing one factor by a proper subgraph of the same chromatic number would give a proper subgraph of the whole join with unchanged chromatic number. For a graph $H$, let $\vartheta(H)$ be the minimum number of cliques whose vertex sets cover $V(H)$; overlaps may be discarded, so this equals the minimum number in a clique partition. We shall repeatedly use
\begin{equation}
\vartheta(\comp{G})=\chi(G).
\end{equation}

\subsection{Critical-graph structure and edge bounds}
We use the following standard results in the forms recorded and applied by Cranston~\cite{Cranston2025}.

\begin{theorem}[Gallai, iterated join form]\label{thm:gallai}
Let $G$ be an $r$-critical graph on $n\le 2r-2$ vertices. Then
\[
G=G_1\join\cdots\join G_t,\qquad t\ge 2,
\]
where $G_i$ is $k_i$-critical, $|V(G_i)|\ge 2k_i-1$, and $\sum_i k_i=r$.
\end{theorem}

\begin{theorem}[Gallai--Kostochka--Stiebitz]\label{thm:gks}
If $G$ is an $n$-vertex $r$-critical graph, $r\ge 4$, and $r+2\le n\le 2r-1$, then
\begin{equation}
\e(G)\ge \frac{(r-1)n+(n-r)(2r-n)-2}{2}.
\end{equation}
\end{theorem}
Gallai proved the range $n\le 2r-2$~\cite{Gallai1963}; Kostochka and Stiebitz extended the endpoint $n=2r-1$~\cite{KostochkaStiebitz1999}.

\begin{theorem}[Bar\'at--T\'oth / Kostochka--Stiebitz]\label{thm:barat-excess}
Let $H$ be a $k$-critical graph with $k\ge 4$. If $H$ contains no subdivision of $K_k$, then
\begin{equation}
2\e(H)\ge (k-1)|V(H)|+(2k-6),
\end{equation}
that is, $\e(H)\ge (k-1)|V(H)|/2+k-3$.
\end{theorem}

\begin{theorem}[Bar\'at--T\'oth]\label{thm:small-topological}
Every $r$-critical graph on at most $r+4$ vertices contains a subdivision of $K_r$.
\end{theorem}

\begin{theorem}[Cranston: order ranges]\label{thm:cranston-ranges}
Let $G$ be $r$-critical. Then Albertson's inequality holds whenever
\[
1.228r\le |V(G)|\le 1.768r,
\]
and, for $r\ge 15$, whenever
\[
|V(G)|\ge 2.8118r.
\]
\end{theorem}

\begin{remark}\label{rem:cranston-endpoints}
The introductory Theorem~1 of arXiv:2512.08020v1 prints the endpoints $1.212r$ and $2.812r$, while the proved intermediate theorem in Section~3 is $1.228r\le |V(G)|\le 1.768r$, and the proved large-order theorem uses $2.8118r$. We use the latter proved statements only~\cite{Cranston2025}.
\end{remark}

\begin{theorem}[Cranston: residual cases for $r\le 26$]\label{thm:cranston-residual}
Albertson's conjecture holds for every $r\le 24$. Moreover, if $G$ is $r$-critical, $r\le 26$, and $\crn(G)<\crn(K_r)$, then
\[
(r,|V(G)|)\in\{(25,48),(26,50),(26,51)\}.
\]
\end{theorem}

\subsection{Essential immersions and a one-unit degree gain}
Following Oporowski and Zhao~\cite{OporowskiZhao2009}, an immersion of $H$ in $G$ maps vertices of $H$ injectively to branch vertices of $G$ and distinct edges of $H$ to pairwise edge-disjoint paths. It is \emph{essential} if paths corresponding to nonadjacent target edges are vertex-disjoint. Oporowski and Zhao proved the crossing-monotonicity we need.

\begin{proposition}[Oporowski--Zhao]\label{prop:essential-crossing}
If $H$ is essentially immersed in $G$, then $\crn(H)\le \crn(G)$.
\end{proposition}

The next lemma generalizes the Kempe-chain mechanism used for the $K_6$ $v$-immersion in Oporowski--Zhao, Lemma~2.3~\cite{OporowskiZhao2009}.

\begin{lemma}[Kempe essential-immersion lemma]\label{lem:kempe}
Let $G$ be a $k$-critical graph. If $G$ has a vertex $v$ of degree $k-1$, then $G$ contains an essential immersion of $K_k$. The immersion can be chosen branch-clean: no branch vertex is internal to an immersed edge-path.
\end{lemma}
\begin{proof}
The case $k=1$ is trivial, so assume $k\ge 2$. Since $G$ is $k$-critical, $G-v$ has a proper $(k-1)$-coloring $c$. Because $d(v)=k-1$ and the coloring cannot be extended to $v$, the neighbors of $v$ receive all $k-1$ colors exactly once. Write
\[
N(v)=\{x_1,\ldots,x_{k-1}\},\qquad c(x_i)=i.
\]
Fix distinct $i,j$. In the subgraph of $G-v$ induced by colors $i$ and $j$, the vertices $x_i$ and $x_j$ lie in the same component. Otherwise, interchange colors $i$ and $j$ on the component containing $x_i$; since $x_i,x_j$ are the unique neighbors of $v$ with those colors, color $i$ disappears from $N(v)$ and can be assigned to $v$, a contradiction. Choose a simple bichromatic path $P_{ij}$ from $x_i$ to $x_j$.

Use $v,x_1,\ldots,x_{k-1}$ as branch vertices. Represent $vx_i$ by the single edge $vx_i$ and $x_ix_j$ by $P_{ij}$. The paths are pairwise edge-disjoint: every edge of $P_{ij}$ has endpoint-color set $\{i,j\}$, so it cannot belong to $P_{pq}$ unless $\{i,j\}=\{p,q\}$; the spokes are distinct and lie outside $G-v$.

If target edges $x_ix_j$ and $x_px_q$ are nonadjacent, then $\{i,j\}\cap\{p,q\}=\varnothing$, so their paths use disjoint color classes and are vertex-disjoint. If $vx_i$ and $x_px_q$ are nonadjacent, then $i\notin\{p,q\}$; the path $P_{pq}\subseteq G-v$ contains neither $v$ nor $x_i$. Thus the immersion is essential. Finally, $P_{ij}$ uses only colors $i,j$, and among the branch vertices the only vertices of those colors are its endpoints. Hence no branch vertex occurs internally.
\end{proof}

\begin{corollary}[Minimum-degree strengthening]\label{cor:degree-gain}
If a $k$-critical graph $G$ is a counterexample to Albertson's conjecture at $r=k$, then
\begin{equation}
\delta(G)\ge k,
\end{equation}
and consequently
\begin{equation}
\e(G)\ge \left\lceil\frac{k|V(G)|}{2}\right\rceil.
\end{equation}
\end{corollary}
\begin{proof}
By Lemma~\ref{lem:min-degree}, $\delta(G)\ge k-1$. Equality would trigger Lemma~\ref{lem:kempe}; then Proposition~\ref{prop:essential-crossing} would give $\crn(G)\ge \crn(K_k)$, contrary to the counterexample assumption. The edge bound follows by the handshaking lemma and integrality.
\end{proof}

\subsection{A linear crossing inequality and induced-subgraph averaging}
B\"ungener and Kaufmann proved the following unconditional inequality~\cite{BungenerKaufmann2024}.

\begin{theorem}[B\"ungener--Kaufmann]\label{thm:bk}
If $F$ is a finite simple graph with $N\ge 3$ vertices and $M$ edges, then
\begin{equation}
\crn(F)\ge 5M-\frac{203}{9}(N-2).
\end{equation}
\end{theorem}

The following induced-subgraph averaging inequality is Cranston's inequality~(1) in~\cite{Cranston2025}; we reproduce the short double-counting proof for completeness.

\begin{lemma}[Induced-subgraph averaging]\label{lem:sampling}
Let $G$ have $n$ vertices and $m$ edges, and let $4\le q\le n$. Then
\begin{equation}
\crn(G)\ge B_q(n,m):=5m\frac{(n-2)(n-3)}{(q-2)(q-3)}-
\frac{203n(n-1)(n-2)(n-3)}{9q(q-1)(q-3)}.
\end{equation}
\end{lemma}
\begin{proof}
Fix a crossing-minimal drawing of $G$ in general position. We may assume adjacent edges do not cross, no edge crosses itself, and no three edges cross at one interior point; hence every crossing has four distinct endpoints. For each $q$-set $S\subseteq V(G)$, let $M_S=\e(G[S])$ and let $X_S$ be the number of crossings surviving in the inherited drawing. By Theorem~\ref{thm:bk},
\[
X_S\ge \crn(G[S])\ge 5M_S-\frac{203}{9}(q-2).
\]
Summing over all $S$, each crossing is counted $\binom{n-4}{q-4}$ times and each edge is counted $\binom{n-2}{q-2}$ times. Therefore
\[
\crn(G)\binom{n-4}{q-4}\ge 5m\binom{n-2}{q-2}-\frac{203}{9}(q-2)\binom nq.
\]
Division by $\binom{n-4}{q-4}$ gives (9).
\end{proof}

\subsection{Matching, coloring, and triangle-free tools}
We shall also use four standard results. They are stated here to make every external input explicit.

\begin{theorem}[Stehl\'ik~\cite{Stehlik2003}]\label{thm:stehlik}
Let $G$ be a $k$-critical graph whose complement is connected. For every $v\in V(G)$, the graph $G-v$ has a proper $(k-1)$-coloring in which every color class has at least two vertices.
\end{theorem}

\begin{theorem}[Andr\'asfai--Erd\H{o}s--S\'os~\cite{AndrasfaiErdosSos1974}]\label{thm:aes}
If a triangle-free graph $H$ on $N$ vertices satisfies $\delta(H)>2N/5$, then $H$ is bipartite.
\end{theorem}

\begin{theorem}[Tutte~\cite{Tutte1947}]\label{thm:tutte}
A graph $F$ has a perfect matching if and only if
\[
\odd(F-S)\le |S|\qquad\text{for every }S\subseteq V(F),
\]
where $\odd(J)$ denotes the number of odd components of $J$.
\end{theorem}

\begin{theorem}[Hall~\cite{Hall1935}]\label{thm:hall}
A bipartite graph with bipartition $(A,B)$ has a matching saturating $A$ if and only if $|N(Q)|\ge |Q|$ for every $Q\subseteq A$.
\end{theorem}

\begin{theorem}[Rabern~\cite{Rabern2014}]\label{thm:rabern}
Every graph $G$ satisfies
\[
\chi(G)\le \max\left\{\omega(G),\,\Delta(G)-1,\,
\left\lceil\frac{15+\sqrt{48|V(G)|+73}}{4}\right\rceil\right\}.
\]
\end{theorem}

\section{\texorpdfstring{The cases $r=25$ and $r=26$}{The cases r=25 and r=26}}
The cases 25 and 26 require no immersion theory. We first isolate the edge estimate at order $2r-2$.

\begin{lemma}[Edge bound at order $2r-2$]\label{lem:2rminus2}
Let $r\ge 5$, and let $G$ be an $r$-critical graph on $2r-2$ vertices. If $G$ contains no subdivision of $K_r$ and $m=\e(G)$, then
\begin{equation}
m\ge r^2-\frac r2-4.
\end{equation}
In particular, $r=25$ implies $m\ge 609$, and $r=26$ implies $m\ge 659$.
\end{lemma}
\begin{proof}
Apply Theorem~\ref{thm:gallai} and write
\[
G=G_1\join\cdots\join G_t,\qquad G_i\text{ is }k_i\text{-critical},\qquad |V(G_i)|\ge 2k_i-1,\qquad \sum_i k_i=r.
\]
At least one block $G_i$ contains no subdivision of $K_{k_i}$. Indeed, otherwise choose such a subdivision in every block and use direct join edges between branch vertices in different blocks; this gives a subdivision of $K_r$, contrary to the hypothesis.

Fix such a block $J$, and put
\[
k=\chi(J),\qquad a=|V(J)|,\qquad W=G-V(J),\qquad b=2r-2-a,\qquad s=\chi(W)=r-k.
\]
The cases $k=1,2$ are $K_1,K_2$, and every 3-critical graph is an odd cycle and hence a subdivision of $K_3$. Thus $4\le k\le r-1$. Gallai's block-size bound gives
\[
2k-1\le a\le r+k-2,
\]
where the upper bound follows from $b\ge \chi(W)=r-k$. By Theorem~\ref{thm:barat-excess},
\[
\e(J)\ge \frac{a(k-1)}2+k-3.
\]
For a vertex in a block $G_i\subseteq W$, its degree inside $W$ is at least
\[
(k_i-1)+\sum_{\ell\ne i,J}|V(G_\ell)|\ge (k_i-1)+\sum_{\ell\ne i,J}k_\ell=s-1,
\]
so $\e(W)\ge b(s-1)/2$. All $ab$ cross edges are present. Hence
\[
\begin{aligned}
m&\ge \frac{a(k-1)}2+(k-3)+\frac{(2r-2-a)(r-k-1)}2+a(2r-2-a)\\
&=:F_r(k,a).
\end{aligned}
\]
For fixed $k$, this is concave in $a$, so the minimum on $2k-1\le a\le r+k-2$ occurs at an endpoint. Direct substitution gives
\[
F_r(k,2k-1)=-2k^2+2kr+k+r^2-\frac{7r}{2}-1,
\]
\[
F_r(k,r+k-2)=-\frac{kr}{2}+2k+\frac{3r^2}{2}-3r-2.
\]
The first expression is concave in $k$, so its minimum on $4\le k\le r-1$ occurs at $k=4$ or $k=r-1$; the values are
\[
r^2+\frac{9r}{2}-29,\qquad r^2-\frac r2-4,
\]
and their difference is $5(r-5)\ge 0$. The second expression decreases in $k$ and equals $r^2-r/2-4$ at $k=r-1$. This proves (10). Integrality gives the stated values.
\end{proof}

\begin{theorem}\label{thm:r25}
Albertson's conjecture holds for $r=25$.
\end{theorem}
\begin{proof}
Suppose $G_0$ is a counterexample. By Lemma~\ref{lem:critical-reduction}, it contains a 25-critical subgraph $G$ with $\crn(G)<\crn(K_{25})$. By Theorem~\ref{thm:cranston-residual}, $|V(G)|=48$. The graph $G$ contains no subdivision of $K_{25}$ by Lemma~\ref{lem:subdivision}; since $48=2\cdot25-2$, Lemma~\ref{lem:2rminus2} gives $\e(G)\ge 609$. With $q=23$,
\[
\crn(G)\ge B_{23}(48,609)=\frac{101181}{22}=4356+\frac{5349}{22}>4356\ge \crn(K_{25}),
\]
a contradiction.
\end{proof}

\begin{theorem}\label{thm:r26}
Albertson's conjecture holds for $r=26$.
\end{theorem}
\begin{proof}
Take a 26-critical counterexample $G$. By Theorem~\ref{thm:cranston-residual}, $|V(G)|\in\{50,51\}$, and by Lemma~\ref{lem:subdivision} it contains no subdivision of $K_{26}$.

If $|V(G)|=50$, Lemma~\ref{lem:2rminus2} gives $\e(G)\ge 659$. Thus
\[
\crn(G)\ge B_{24}(50,659)=\frac{255109420}{47817}=5148+\frac{8947504}{47817}>5148.
\]
If $|V(G)|=51$, Theorem~\ref{thm:barat-excess} gives
\[
2\e(G)\ge 25\cdot 51+(52-6)=1321.
\]
Since the left side is even, $\e(G)\ge 661$, and
\[
\crn(G)\ge B_{24}(51,661)=\frac{11738860}{2277}=5148+\frac{16864}{2277}>5148.
\]
Both contradict $\crn(K_{26})\le 5148$.
\end{proof}

\section{A strengthened Gallai-join estimate}
The branch-clean conclusion of Lemma~\ref{lem:kempe} permits a useful refinement of Gallai's decomposition that will be used for $r=27,28,29$.

\begin{lemma}[Join edge bound]\label{lem:join-edge}
Let $G$ be an $r$-critical counterexample, and suppose either $|V(G)|\le 2r-2$ or $\comp{G}$ is disconnected. Put $n=|V(G)|$. In a terminal Gallai join decomposition there is a block $J$ with
\[
k:=\chi(J)\ge 4,\qquad a:=|V(J)|,\qquad \delta(J)\ge k.
\]
Writing $W=G-V(J)$, $b=n-a$, and $s=r-k$, one has
\begin{equation}
\e(G)\ge F_{r,n}(a,k):=\frac{ak}{2}+\frac{(n-a)(r-k-1)}{2}+a(n-a),
\end{equation}
with
\begin{equation}
2k-1\le a\le n-r+k.
\end{equation}
If $n=r+p\le 2r-2$ and $r-12\le p\le r-2$, then
\begin{equation}
\e(G)\ge \left\lceil\frac{r^2+2pr-p^2-r+1}{2}\right\rceil.
\end{equation}
Finally, if $n=2r-1$ and $\comp{G}$ is disconnected, then
\begin{equation}
\e(G)\ge r^2-1.
\end{equation}
\end{lemma}
\begin{proof}
Whenever $\comp{G}$ is disconnected, $G$ is a nontrivial join. Iterate the join decomposition until every terminal block has connected complement. If every terminal $k_i$-critical block had a vertex of degree $k_i-1$, Lemma~\ref{lem:kempe} would supply a branch-clean essential immersion of $K_{k_i}$ in each block. Use local immersed paths inside the blocks and direct join edges between branch vertices in different blocks. The resulting paths form an essential immersion of $K_{\sum_i k_i}=K_r$: branch-cleanliness handles precisely the case in which a local path and a nonadjacent cross-block edge could otherwise meet at an unrelated branch vertex. This contradicts Proposition~\ref{prop:essential-crossing}. Hence some terminal block $J$ has $\delta(J)\ge k$. Since every $k$-critical graph with $k\le 3$ has a vertex of degree $k-1$, we have $k\ge 4$.

Now $\e(J)\ge ak/2$. For a vertex in a terminal block of $W$, its degree inside $W$ is at least its critical degree plus all vertices of the other blocks; using $|V(G_i)|\ge k_i$ gives $\delta(W)\ge s-1$. Thus $\e(W)\ge b(s-1)/2$, and all $ab$ cross edges are present, proving (11). The lower bound in (12) is Gallai's terminal-block bound; the upper bound follows from $b\ge \chi(W)=r-k$.

For fixed $k$, $F_{r,n}(a,k)$ is concave in $a$, so its minimum occurs at an endpoint. Put $n=r+p$. Feasibility gives $4\le k\le p+1$, and
\[
F_{r,r+p}(2k-1,k)=\frac{-4k^2+3kp+kr+8k+pr-3p+r^2-2r-3}{2},
\]
\[
F_{r,r+p}(p+k,k)=\frac{-kp+k+2pr+r^2-r}{2}.
\]
The second decreases in $k$ for $p>1$; the first is concave in $k$, so its minimum occurs at $k=4$ or $k=p+1$. Moreover
\[
F_{r,r+p}(7,4)-F_{r,r+p}(2p+1,p+1)=\frac{(p-3)(p-r+12)}{2}\ge 0
\]
under the stated hypothesis. Thus the global minimum is attained at $k=p+1$, where the two $a$-endpoints coincide, yielding (13) after integrality.

For $n=2r-1$ with disconnected complement, the same computation gives $2k-1\le a\le r+k-1$. At the two endpoints,
\[
F_{r,2r-1}(2k-1,k)=\frac{-4k^2+4kr+5k+2r^2-6r}{2},
\]
\[
F_{r,2r-1}(r+k-1,k)=\frac{-kr+2k+3r^2-3r}{2}.
\]
The first is concave in $k$ and exceeds $r^2-1$ at both $k=4$ and $k=r-1$ (by $5r-21$ and $(3r-7)/2$, respectively); the second decreases in $k$ and equals $r^2-1$ at $k=r-1$. This proves (14).
\end{proof}

\section{\texorpdfstring{The case $r=27$}{The case r=27}}
Assume for contradiction that $G$ is a 27-critical counterexample, and write $n=|V(G)|$, $m=\e(G)$. By Corollary~\ref{cor:degree-gain},
\begin{equation}
\delta(G)\ge 27,\qquad m\ge \left\lceil\frac{27n}{2}\right\rceil.
\end{equation}
We exclude all possible orders.

\subsection{\texorpdfstring{Orders $27\le n\le 51$}{Orders 27 <= n <= 51}}
If $27\le n\le 31$, then $n\le 27+4$, so Theorem~\ref{thm:small-topological} gives a subdivision of $K_{27}$, contradicting Lemma~\ref{lem:subdivision}.

For $n=32,33$, Theorem~\ref{thm:gks} gives $m\ge 470,491$, respectively. With $q=12$,
\[
B_{12}(32,470)=6084+\frac{219946}{891}>6084,
\qquad
B_{12}(33,491)=6084+\frac{17237}{27}>6084.
\]
For $34\le n\le 47$,
\[
1.228\cdot 27=33.156\le n\le 47<47.736=1.768\cdot 27,
\]
so Theorem~\ref{thm:cranston-ranges} applies.

For $48\le n\le 51$, Theorem~\ref{thm:gks} and $q=22$ give the exact margins
\[
\begin{array}{c|c|c}
 n & \text{lower bound on }m & B_{22}(n,m)-6084\\\hline
48&686&125579/209\\
49&691&1906889/3762\\
50&695&34064/99\\
51&698&66524/627
\end{array}
\]
so all these orders are impossible.

\subsection{\texorpdfstring{The boundary order $n=52$}{The boundary order n=52}}
At $n=52=2\cdot 27-2$, Lemma~\ref{lem:join-edge} gives a particularly simple bound. Here $p=25$, so (13) yields
\[
m\ge \left\lceil\frac{27^2+2\cdot25\cdot27-25^2-27+1}{2}\right\rceil=714.
\]
Equivalently, one can see this directly from the endpoint computation in the proof of Lemma~\ref{lem:join-edge}: the distinguished block has $4\le k\le 26$ and the minimum occurs at $k=26$. With $q=22$,
\[
B_{22}(52,714)=6084+\frac{447071}{1881}>6084,
\]
so $n=52$ is impossible.

\subsection{\texorpdfstring{Orders $53\le n\le 75$}{Orders 53 <= n <= 75}}
For $n=53$, (15) gives $m\ge 716$, and
\[
B_{24}(53,716)=6084+\frac{54647}{47817}>6084.
\]
This is the tightest displayed numerical comparison in the $r=27$ argument.

For $54\le n\le 61$, it suffices to use the real lower bound $m\ge 27n/2$. Substituting into (9) with $q=24$ gives
\begin{equation}
F_{24}(n):=B_{24}\left(n,\frac{27n}{2}\right)
=-\frac{n(n-3)(n-2)(2233n-169903)}{1147608}.
\end{equation}
Its second derivative is
\[
F_{24}''(n)=-\frac{13398n^2-543204n+862913}{573804}.
\]
The numerator is positive for $n\ge 54$ because it is positive at 54 and has positive, increasing derivative there. Thus $F_{24}$ is concave on $[54,61]$, so its minimum occurs at an endpoint. Exact evaluation gives
\[
F_{24}(54)=6084+\frac{125177}{1771}>6084,
\qquad
F_{24}(61)=6084+\frac{4205909}{95634}>6084.
\]
Hence every integer $54\le n\le 61$ is excluded.

For $62\le n\le 75$, take $q=28$. Again using $m\ge 27n/2$,
\begin{equation}
F_{28}(n):=B_{28}\left(n,\frac{27n}{2}\right)
=-\frac{n(n-3)(n-2)(377n-33182)}{315900}.
\end{equation}
Now
\[
F_{28}''(n)=-\frac{2262n^2-105201n+168172}{157950}.
\]
The numerator is positive for $n\ge 62$, so $F_{28}$ is concave on $[62,75]$. At the endpoints,
\[
F_{28}(62)=6084+\frac{3845404}{5265}>6084,
\qquad
F_{28}(75)=6084+\frac{4594}{117}>6084.
\]
Thus every integer $62\le n\le 75$ is excluded.

Finally, Theorem~\ref{thm:cranston-ranges} applies for $n\ge 2.8118\cdot27=75.9186$, so every integer $n\ge 76$ is excluded.

\begin{theorem}\label{thm:r27}
Albertson's conjecture holds for $r=27$.
\end{theorem}
\begin{proof}
The preceding subsections exclude every possible order of a 27-critical counterexample. The critical reduction Lemma~\ref{lem:critical-reduction} then excludes every graph of chromatic number at least 27.
\end{proof}

\section{\texorpdfstring{The case $r=28$}{The case r=28}}
Assume for contradiction that $G$ is a 28-critical counterexample, and write $n=|V(G)|$ and $m=\e(G)$. By Corollary~\ref{cor:degree-gain},
\begin{equation}
\delta(G)\ge 28,\qquad m\ge 14n.
\end{equation}

\subsection{Reduction to order 55}
The small-order Theorem~\ref{thm:small-topological} excludes $n\le 32$. For $n=33,34$, Theorems~\ref{thm:gks} and~\ref{lem:sampling} give
\[
\begin{array}{c|c|c|c}
n&m_{\min}&q&B_q(n,m_{\min})-7098\\\hline
33&502&12&5204/27\\
34&524&12&52502/81
\end{array}
\]
so both orders are impossible. Cranston's intermediate range excludes every $35\le n\le 49$. For $50\le n\le 54$, Lemma~\ref{lem:join-edge} gives
\[
\begin{array}{c|c|c|c}
n&m_{\min}&q&B_q(n,m_{\min})-7098\\\hline
50&753&20&1318838/969\\
51&758&20&1058078/969\\
52&763&21&1302658/1539\\
53&766&22&1025137/1881\\
54&769&23&2335138/8855
\end{array}
\]
All margins are positive. For $56\le n\le 78$, the degree bound $m\ge 14n$ and the exact sampling values in Table~\ref{tab:r28} exclude every integer order. Cranston's large-order range excludes $n\ge 79$, since $2.8118\cdot 28=78.7304$.

Thus only $n=55$ remains. Here $m\ge 770$, and at $q=24$,
\[
B_{24}(55,770)-7098=-\frac{1339}{23},\qquad
B_{24}(55,771)-7098=-\frac{150839}{5313},
\]
whereas
\[
B_{24}(55,772)-7098=\frac{7631}{5313}>0.
\]
Hence
\begin{equation}
n=55,\qquad m\in\{770,771\}.
\end{equation}
If $\comp{G}$ were disconnected, (14) would give $m\ge 28^2-1=783$, impossible. Therefore $\comp{G}$ is connected.

\subsection{The complement at order 55}
Put $H=\comp{G}$. By Theorem~\ref{thm:stehlik}, for every $v\in V(G)$ the graph $G-v$ has a 27-coloring in which every color class has at least two vertices. Since $|V(G-v)|=54$, every class has size exactly two. Equivalently,
\begin{equation}
H-v\quad\text{has a perfect matching for every }v\in V(H),
\end{equation}
so $H$ is factor-critical.

We shall use the following consequence repeatedly. Call a factor-critical complement in this situation \emph{anti-tight} if, for every $v$ and every perfect matching $M$ of $H-v$, no matching edge $xy\in M$ has both endpoints in $N_H(v)$. Indeed, otherwise $\{v,x,y\}$ is a triangle of $H$ and, together with the remaining 26 matching edges, partitions $V(H)$ into 27 cliques. By (3), this would imply $\chi(G)\le 27$, a contradiction. Hence $H$ is anti-tight.

Since $\binom{55}{2}=1485$, (19) gives $\e(H)=715$ or $714$. Also $\Delta(H)\le 26$ by (18). If $m=770$, then $2\e(H)=1430=55\cdot26$, so $H$ is 26-regular. If $m=771$, the degree sum is 1428, exactly two below $55\cdot26$; hence the degree sequence is either
\[
(26^{54},24)\qquad\text{or}\qquad(26^{53},25^2).
\]
In every case
\begin{equation}
\delta(H)\ge 24.
\end{equation}

\begin{lemma}\label{lem:r28-triangle}
The graph $H$ contains a triangle. Moreover, for every triangle $T\subseteq H$, the graph $H-T$ has no perfect matching.
\end{lemma}
\begin{proof}
If $H$ were triangle-free, then (21) and $24>2\cdot55/5$ would make $H$ bipartite by Theorem~\ref{thm:aes}. No bipartite graph of odd order is factor-critical, contradicting (20).

Let $T=\{x,y,z\}$ be a triangle. If $H-T$ had a perfect matching $M$, then $M\cup\{yz\}$ would be a perfect matching of $H-x$ containing an edge whose two endpoints both lie in $N_H(x)$, contradicting anti-tightness.
\end{proof}

\subsection{Tutte barriers}
Fix a triangle $T$. By Lemma~\ref{lem:r28-triangle}, $H-T$ has no perfect matching. Tutte's theorem gives a set $R\subseteq V(H)\setminus T$ such that
\[
\odd(H-(T\cup R))>|R|.
\]
Since $|H-T|=52$ is even, the defect has even parity and is at least two. Put
\[
S=T\cup R,\qquad s=|S|,
\]
so
\begin{equation}
\odd(H-S)\ge s-1.
\end{equation}
Every component $C$ of $H-S$ satisfies
\begin{equation}
|C|\ge \delta(H)-s+1,
\end{equation}
because a vertex of $C$ has at most $|C|-1+s$ neighbors in $H$.

Combining (22) and (23) with parity of the odd components gives the following complete list:
\[
\begin{array}{c|c}
\text{degree case}&\text{possible }s\\\hline
H\text{ 26-regular}&3,26,27,28\\
\delta(H)\ge 24&3,24,25,26,27,28
\end{array}
\]
For example, if $L$ is the least odd integer at least $\delta(H)-s+1$, then $(s-1)L\le 55-s$ is necessary, and direct checking gives exactly the displayed values.

Let $X=V(H)\setminus S$, $e_S=\e(H[S])$, and $e_X=\e(H[X])$. Since there are at least $s-1$ odd components,
\begin{equation}
e_X\le \binom{57-2s}{2}\qquad(s\le 28).
\end{equation}
In the 26-regular case, comparison of the degree sums on $S$ and $X$ gives
\begin{equation}
e_S=e_X+13(2s-55).
\end{equation}
Thus $s=26,27$ force $e_S<0$ by (24).

In the 714-edge case put $a(v)=26-d_H(v)$, so $\sum_v a(v)=2$, and let $A_S,A_X$ be the total deficits on the two sides. Then degree-sum comparison gives
\begin{equation}
e_S=e_X+13(2s-55)+\frac{A_X-A_S}{2}.
\end{equation}
Since $|A_X-A_S|\le 2$, the largest possible right sides for $s=24,25,26,27$ are, respectively,
\[
36-91+1,\qquad21-65+1,\qquad10-39+1,\qquad3-13+1,
\]
all negative. Consequently only
\begin{equation}
s\in\{3,28\}
\end{equation}
can remain.

If $s=3$, (22) gives at least two odd components of $H-S$. Every component has order at least 22, and every odd component therefore has order at least 23. Four odd components cannot fit in 52 vertices, and after two odd components no additional even component can occur because $23+23+22>52$. Thus $H-S$ has exactly two odd components, of orders $(23,29)$ or $(25,27)$. For a component $C$ of order $c$, every vertex has at least 21 neighbors inside $C$, so in $G[C]$ it has degree at most $c-22$. Hence $\chi(G[C])\le c-21$ by the greedy bound. The two components are completely joined in $G$, and in either size pattern they require at most 10 colors in total; the three vertices of $S$ use at most three more. Thus $\chi(G)\le 13$, a contradiction.

It remains to exclude $s=28$. Then $|X|=27$ and (22) forces every vertex of $X$ to be an isolated component of $H-S$, so $H[X]$ is edgeless. In the regular case $e_S=715-27\cdot26=13$. In the 714-edge case, if $A_X=\sum_{x\in X}(26-d_H(x))$, then $e_S=714-(702-A_X)=12+A_X$. Thus
\begin{equation}
12\le e_S\le 14.
\end{equation}
Choose an edge $yz\in H[S]$. For any edge in a graph with $e_S$ edges, $d_S(y)+d_S(z)\le e_S+1\le 15$. Since $d_H(y),d_H(z)\ge 24$,
\[
d_X(y)+d_X(z)\ge 48-15=33>27,
\]
so $y,z$ have a common neighbor $x\in X$ and $xyz$ is a triangle.

Set $X'=X\setminus\{x\}$ and $S'=S\setminus\{y,z\}$. Both have 26 vertices. Every vertex of $X'$ had at least 24 neighbors in $S$ and loses at most $y,z$, so its degree into $S'$ is at least 22. Hall is therefore automatic for subsets of $X'$ of size at most 22. Suppose $Q\subseteq X'$ has size $t\ge 23$ and $|N(Q)|\le t-1$. Then $S'\setminus N(Q)$ contains at least $27-t$ vertices. Any such vertex $u$ has all its original $X$-neighbors among $x$ and the $26-t$ vertices of $X'\setminus Q$, so
\[
d_X(u)\le 27-t\le 4.
\]
But (28) and $d_H(u)\ge 24$ give
\[
d_X(u)=d_H(u)-d_S(u)\ge 24-e_S\ge 10,
\]
a contradiction. Hence Hall gives a perfect matching of $H-\{x,y,z\}$, contradicting Lemma~\ref{lem:r28-triangle}. This excludes $s=28$.

\begin{theorem}\label{thm:r28}
Albertson's conjecture holds for $r=28$.
\end{theorem}
\begin{proof}
The preceding reduction leaves only order 55, and the complement analysis excludes both possible edge counts at that order. Hence no 28-critical counterexample exists; Lemma~\ref{lem:critical-reduction} completes the proof.
\end{proof}

\section{\texorpdfstring{The case $r=29$}{The case r=29}}
Assume for contradiction that $G$ is a 29-critical counterexample, with $n=|V(G)|$ and $m=\e(G)$. Then
\begin{equation}
\delta(G)\ge 29,\qquad m\ge \left\lceil\frac{29n}{2}\right\rceil.
\end{equation}

\subsection{Exact order reduction}
The small-order theorem excludes $n\le 33$. For $n=34,35$, Theorems~\ref{thm:gks} and~\ref{lem:sampling} give
\[
\begin{array}{c|c|c|c}
n&m_{\min}&q&B_q(n,m_{\min})-8281\\\hline
34&535&12&5783/81\\
35&558&12&48295/81
\end{array}
\]
Cranston excludes $36\le n\le 51$. For $52\le n\le 56$, Lemma~\ref{lem:join-edge} gives
\[
\begin{array}{c|c|c|c}
n&m_{\min}&q&B_q(n,m_{\min})-8281\\\hline
52&809&20&4125898/2907\\
53&815&20&21646/19\\
54&819&21&236002/285\\
55&823&22&10088/19\\
56&825&23&58301/322
\end{array}
\]
For $59\le n\le 81$, (29) and the exact values in Table~\ref{tab:r29} exclude every order. Cranston excludes $n\ge 82$ because $2.8118\cdot29=81.5422$. Hence only $n=57,58$ remain.

At $n=57$, $m\ge 827$, while
\[
B_{24}(57,831)-8281=-\frac{2706}{161},\qquad
B_{24}(57,832)-8281=\frac{2469}{161}>0.
\]
Thus
\begin{equation}
827\le m\le 831.
\end{equation}
If $\comp{G}$ were disconnected, (14) would give $m\ge 29^2-1=840$, impossible; hence $\comp{G}$ is connected.

At $n=58$, $m\ge 841$, and
\[
B_{24}(58,841)-8281=-\frac{37631}{621},\qquad
B_{24}(58,842)-8281=-\frac{16931}{621},
\]
whereas
\[
B_{24}(58,843)-8281=\frac{3769}{621}>0.
\]
Thus only $m=841,842$ are possible at order 58.

\subsection{The 57-vertex cases}
Put $H=\comp{G}$. Exactly as in the $r=28$ argument, Theorem~\ref{thm:stehlik} makes $H$ factor-critical, and the matching/clique-partition argument makes it anti-tight. Since $\binom{57}{2}=1596$ and $\Delta(H)\le 27$, define
\[
a(v)=27-d_H(v),\qquad A=\sum_v a(v)=57\cdot27-2\e(H).
\]
For the five edge counts in (30), the corresponding values are
\begin{equation}
A\in\{1,3,5,7,9\}.
\end{equation}

\begin{lemma}\label{lem:r29-57-triangle}
For every value in (31), $H$ contains a triangle.
\end{lemma}
\begin{proof}
If $A=1$, then $\delta(H)\ge 26$; if $A=3$, then $\delta(H)\ge 24$. In either case $\delta(H)>2\cdot57/5$, so a triangle-free $H$ would be bipartite by Theorem~\ref{thm:aes}, impossible for a factor-critical graph of odd order.

Suppose now $A\in\{5,7,9\}$ and $H$ is triangle-free. Since $H$ is factor-critical it is not bipartite, so Theorem~\ref{thm:aes} gives a vertex $v$ with $d_H(v)\le 22$. Put $D=a(v)\ge 5$ and $R_0=A-D$.

If $R_0\le 3$, every vertex of $H-v$ has degree at least $26-R_0\ge 23$. Hence $H-v$ is bipartite. It has a perfect matching, so its parts $P,Q$ both have 28 vertices. Let $x=|N(v)\cap P|$ and $y=|N(v)\cap Q|$. Both are positive, for otherwise $H$ itself would be bipartite. Since $H$ is triangle-free, there are no edges between $N(v)\cap P$ and $N(v)\cap Q$. For $p\in N(v)\cap P$,
\[
d_{H-v}(p)\ge 26-R_0,\qquad d_{H-v}(p)\le 28-y,
\]
so $y\le 2+R_0$; similarly $x\le 2+R_0$. Therefore
\[
27-A+R_0=d_H(v)=x+y\le 4+2R_0,
\]
which requires $R_0\ge 23-A$. This is impossible for $A=5,7$. For $A=9$, the only remaining possibility is $D=5,R_0=4$.

In this exceptional case, if the remaining deficit 4 is not entirely carried by a neighbor $w$ of $v$, then every vertex of $H-v$ still has degree at least 23 and the preceding bipartite argument applies. Thus necessarily
\[
a(v)=5,\qquad a(w)=4,\qquad vw\in E(H),
\]
and all other vertices have degree 27. Let $K=H-\{v,w\}$. Then $|K|=55$ and $\delta(K)\ge 25>2\cdot55/5$, so $K$ is bipartite. A perfect matching of $H-v$ matches $w$ to a vertex whose deletion makes $K$ perfectly matchable; hence the two parts of $K$ differ in size by one. The same follows from a perfect matching of $H-w$. Thus the parts have sizes 28 and 27, and both $v$ and $w$ have a neighbor in the 28-vertex part.

If $v$ had neighbors in both parts, triangle-freeness and $\delta(K)\ge 25$ would allow at most three neighbors in the 28-part and at most two in the 27-part, contradicting $d_K(v)=21$. Hence all 21 neighbors of $v$ in $K$ lie in the 28-part. Similarly all 22 neighbors of $w$ lie in that same part. Since $vw$ is an edge and $H$ is triangle-free, the two neighbor sets are disjoint, but $21+22>28$, a contradiction.
\end{proof}

Fix a triangle $T$. Anti-tightness implies $H-T$ has no perfect matching, so Tutte gives $S\supseteq T$, $s=|S|$, with
\begin{equation}
\odd(H-S)\ge s-1.
\end{equation}
For $A=1,3,5,7,9$, the corresponding minimum degrees are at least $26,24,22,20,18$. The component-size inequality and parity give the complete possibilities
\[
\begin{array}{c|c}
A&\text{possible }s\\\hline
1&3\text{ or }26\le s\le 29\\
3&3\text{ or }24\le s\le 29\\
5&3\text{ or }22\le s\le 29\\
7&3,4\text{ or }20\le s\le 29\\
9&3,4\text{ or }18\le s\le 29
\end{array}
\]
Let $X=V(H)\setminus S$ and write $A_S,A_X$ for the degree deficits. Degree-sum comparison gives
\begin{equation}
e_S=e_X+\frac{27(2s-57)+(A_X-A_S)}{2},
\end{equation}
while (32) yields
\begin{equation}
e_X\le \binom{59-2s}{2}.
\end{equation}
Using $A_X-A_S\le A$, the right side of (33) is negative for every listed intermediate large value $s\le 28$. Indeed, the resulting upper bound is a convex quadratic in $s$, so its maximum on each interval occurs at an endpoint; for $A=5,7,9$ the endpoint pairs are respectively $(-68,-8)$, $(-55,-7)$, and $(-26,-6)$, while the $A=1,3$ bounds are smaller.

The small values are impossible by coloring. If $H-S$ has $t$ components, then $t\ge s-1$, and a component $C$ of order $c$ satisfies
\[
d_{H[C]}\ge 27-A-s,\qquad \chi(G[C])\le c-27+A+s.
\]
Therefore
\[
\chi(G)\le 57-t(27-A-s).
\]
The worst listed cases are $(A,s,t)=(9,3,2)$, giving 27, and $(9,4,3)$, giving 15; all other cases are stronger. Hence only
\begin{equation}
s=29
\end{equation}
remains.

Now $|X|=28$ and (32) makes $H[X]$ edgeless. For $u\in S$ define
\[
w(u)=a(u)+d_S(u)=27-d_X(u).
\]
From (33),
\begin{equation}
\sum_{u\in S}w(u)=A_S+2e_S=27+A_X\le 36.
\end{equation}
There is an edge $yz\in H[S]$ because $e_S=(27+A_X-A_S)/2\ge 9$. For any such edge,
\[
w(y)+w(z)\le e_S+1+A_S=\frac{29+A}{2}\le 19.
\]
Thus $d_X(y)+d_X(z)\ge 35$, so $y,z$ have a common neighbor $x\in X$.

Delete $x,y,z$. The remaining bipartite graph between $X'=X\setminus\{x\}$ and $S'=S\setminus\{y,z\}$ is $27\times27$. Every vertex of $X'$ has degree at least 16 after deleting $y,z$, so Hall is automatic for subsets of size at most 16. If $Q\subseteq X'$ has size $t\ge 17$ and violates Hall, then $S'\setminus N(Q)$ contains at least $r_0=28-t$ vertices; each has at most $r_0$ neighbors in the original $X$ and hence weight at least $27-r_0$. If $r_0\ge 2$, (36) gives
\[
36\ge r_0(27-r_0)\ge 50,
\]
impossible. If $r_0=1$, some $u\in S'$ has $w(u)\ge 26$. If $d_S(u)=0$, then $w(u)=a(u)\le 9$; otherwise an $S$-neighbor $v$ of $u$ satisfies the edge bound $w(u)+w(v)\le 19$. Both are impossible. Hall therefore gives a perfect matching of $H-\{x,y,z\}$, contradicting anti-tightness.

\begin{proposition}\label{prop:no57}
No 29-critical counterexample has 57 vertices.
\end{proposition}

\subsection{The 58-vertex cases}
If $m=841$, then $2m=58\cdot29$, so $G$ is 29-regular. By Theorem~\ref{thm:rabern},
\[
29=\chi(G)\le \max\{\omega(G),28,18\},
\]
where the last term is the ceiling in Theorem~\ref{thm:rabern} evaluated at 58. Thus $\omega(G)\ge 29$. A $K_{29}$ is then a proper 29-chromatic subgraph of the 58-vertex critical graph $G$, impossible. Hence $m\ne 841$.

Assume from now on
\begin{equation}
|V(G)|=58,\qquad \e(G)=842,
\end{equation}
and put $H=\comp{G}$. The total degree excess of $G$ above 29 is two, so the degree sequence of $G$ is either $(31,29^{57})$ or $(30^2,29^{56})$. Equivalently,
\begin{equation}
\e(H)=811,\qquad (d_H(v))=(26,28^{57})\quad\text{or}\quad(27^2,28^{56}).
\end{equation}
Thus
\begin{equation}
\delta(H)\ge 26,\qquad \varepsilon(v):=28-d_H(v)\ge 0,\qquad \sum_v\varepsilon(v)=2.
\end{equation}

\begin{lemma}\label{lem:r29-58-triangle}
The graph $H$ contains a triangle.
\end{lemma}
\begin{proof}
Suppose not. Choose an edge $xy\in E(G)$. Since $G$ is critical, the proper subgraph $G-xy$ is 28-colorable. Equivalently, $H+xy$ can be partitioned into 28 cliques. Every clique of order at least three in $H+xy$ must contain the newly added edge $xy$; moreover a 4-clique would contain a triangle using only edges of $H$. Hence at most one clique has order three and every other clique has order at most two. Twenty-eight cliques then cover at most $3+27\cdot2=57$ vertices, not 58.
\end{proof}

\subsection{Two disjoint triangles}
\begin{lemma}\label{lem:two-triangles}
The graph $H$ contains two vertex-disjoint triangles.
\end{lemma}
\begin{proof}
Assume otherwise and fix a triangle $T=\{t_1,t_2,t_3\}$. Then $H-T$ is triangle-free and has minimum degree at least 23. Since $23>2\cdot55/5$, Theorem~\ref{thm:aes} makes $H-T$ bipartite; let its parts be $A,B$. Each part is independent in $H$, hence a clique in $G$. A proper $K_{29}$ cannot occur in a 29-critical graph with 58 vertices, so $|A|,|B|\le 28$. Since $|A|+|B|=55$, assume
\begin{equation}
|A|=27,\qquad |B|=28.
\end{equation}
Set
\[
A_i=N_H(t_i)\cap A,\qquad B_i=N_H(t_i)\cap B,\qquad
\alpha=\sum_i|A_i|,\qquad \beta=\sum_i|B_i|.
\]
Let $M$ be the number of missing $A$--$B$ edges and put $\varepsilon_A=\sum_{a\in A}\varepsilon(a)$, $\varepsilon_B=\sum_{b\in B}\varepsilon(b)$. For $a\in A$, the number of missing neighbors in $B$ equals $\varepsilon(a)+d_T(a)$, so $M=\varepsilon_A+\alpha$. For $b\in B$, the corresponding number equals $\varepsilon(b)+d_T(b)-1$, so $M=\varepsilon_B+\beta-28$. Finally,
\[
811=(756-M)+\alpha+\beta+3.
\]
Solving gives
\begin{equation}
M=24+\varepsilon_A+\varepsilon_B,\qquad
\alpha=24+\varepsilon_B\in[24,26],\qquad
\beta=52+\varepsilon_A\in[52,54].
\end{equation}
Every $B_i$ is nonempty, for otherwise $B\cup\{t_i\}$ is an independent 29-set in $H$, hence a proper $K_{29}$ in $G$.

Whenever $A_i\ne\varnothing$, let $F_i$ be the set of $A$--$B$ edges of $H$ joining $A_i$ to $B_i$; each such edge completes a triangle with $t_i$. We claim
\begin{equation}
|F_i|\ge 18.
\end{equation}
A vertex of $A$ misses at most $\varepsilon+d_T\le 5$ vertices of $B$, while a vertex of $B$ misses at most $\varepsilon+d_T-1\le 4$ vertices of $A$. If $x=|A_i|$ and $y=|B_i|$, then
\[
|F_i|\ge \max\{x(y-5),y(x-4)\}.
\]
Also $2+x+y=d_H(t_i)\ge 26$, so $x+y\ge 24$. For positive integers $x,y$ with $x+y\ge 24$, this maximum is at least 18: if $x\le 4$, then $x(y-5)\ge x(19-x)\ge 18$; if $x\ge 5$ and $y\le 5$, then $y(x-4)\ge y(20-y)\ge 19$; and if $x\ge 5,y\ge 6$, then $x(y-5)\ge 18$ (the smallest boundary case is $x=18,y=6$). Thus (42) holds.

For $i\ne j$, every edge of $F_i$ meets every edge of $F_j$, because two disjoint such edges would yield two disjoint triangles. We use the elementary fact that two cross-intersecting edge families in a bipartite graph, each of size at least three, must be stars with a common center. Indeed, if one family contains two disjoint edges, any edge of the other must join opposite endpoints of those two edges, giving at most two choices. Thus both families have matching number one and are stars; two stars of size at least three can be cross-intersecting only with the same center.

Let $q$ be the number of nonempty $A_i$. By (41), $q\ne 0$. If $q=1$, say only $A_1$ is nonempty, then $|B_2|,|B_3|\ge 24$, so $|B_2\cap B_3|\ge 20$. Take $ab\in F_1$ and choose $b_0\in(B_2\cap B_3)\setminus\{b\}$. Then $t_1ab$ and $t_2t_3b_0$ are disjoint triangles, a contradiction.

If $q\ge 2$, all nonempty $F_i$ share a star center. If the center lies in $A$, then $|B_i|\ge 18$ for every nonempty family. Any second vertex of $A_i$ has at least $|B_i|-5\ge 13$ neighbors in $B_i$, producing an edge outside the common star; hence each nonempty $A_i$ is exactly the singleton center. Then $\alpha=q\le 3$, contrary to (41). If the common center lies in $B$, the symmetric argument, using at most four missing $A$-neighbors, makes every corresponding $B_i$ equal to that singleton. For $q=3$ this contradicts $\beta\ge 52$. For $q=2$, each corresponding $t_i$ has $2+|A_i|+1\ge 26$, so $|A_i|\ge 23$ and $\alpha\ge 46$, again impossible.
\end{proof}

Fix two disjoint triangles $T_1,T_2$ and put $U=T_1\cup T_2$. If $H-U$ had a perfect matching, the two triangles and the 26 matching edges would partition $H$ into 28 cliques, contradicting (3). Hence $H-U$ has no perfect matching. Choose an inclusion-minimal Tutte witness $R\subseteq V(H)\setminus U$, and put
\[
S=U\cup R,\qquad s=|S|,\qquad X=V(H)\setminus S.
\]
Because $|H-U|=52$ is even,
\begin{equation}
\odd(H-S)\ge |R|+2=s-4.
\end{equation}
Every component of $H-S$ has order at least $27-s$. Consequently the only possible barrier sizes are
\begin{equation}
s\in\{6,26,27,28,29,30,31\}.
\end{equation}

\begin{lemma}[Minimal-barrier expansion]\label{lem:minimal-barrier}
Every $u\in R$ has neighbors in at least three distinct components of $H-S$. If $s=31$, then $H-S$ consists of 27 isolated vertices and every nonempty $Y\subseteq R$ satisfies
\begin{equation}
|N_X(Y)|\ge |Y|+2.
\end{equation}
If $s=30$ and $H[X]$ has an edge, then $H-S$ has exactly 26 odd components; writing $\mathcal O$ for them, every nonempty $Y\subseteq R$ satisfies
\begin{equation}
|N_{\mathcal O}(Y)|\ge |Y|+2,
\end{equation}
where $N_{\mathcal O}(Y)$ denotes the set of odd components meeting $N_H(Y)$.
\end{lemma}
\begin{proof}
If $u\in R$ met at most two components, adding $u$ back to $H-S$ could reduce the number of odd components by at most one. The current Tutte defect is at least two, so $R\setminus\{u\}$ would still be a witness, contradicting inclusion-minimality.

For the stronger assertions, minimality says that $R\setminus Y$ is not a Tutte witness. If $s=31$, all 27 vertices of $X$ are isolated and the $27-|N_X(Y)|$ untouched vertices remain odd components after adding $Y$ back. Hence
\[
27-|N_X(Y)|\le |R|-|Y|=25-|Y|,
\]
which is (45). If $s=30$ and $H[X]$ has an edge, parity and (43) force exactly 26 odd components. The untouched odd components similarly give $26-|N_{\mathcal O}(Y)|\le 24-|Y|$, proving (46).
\end{proof}

Let $e_S=\e(H[S])$, $e_X=\e(H[X])$, and $E_S=\sum_{v\in S}\varepsilon(v)$, $E_X=\sum_{v\in X}\varepsilon(v)$. Comparing degree sums gives
\begin{equation}
e_S=e_X+28(s-29)+\frac{E_X-E_S}{2}.
\end{equation}
From (43),
\begin{equation}
e_X\le \binom{63-2s}{2}.
\end{equation}
For $s=26,27,28$, equations (47) and (48) give
\[
e_S\le 55-84+1=-28,\qquad e_S\le 36-56+1=-19,\qquad e_S\le 21-28+1=-6,
\]
respectively, impossible. If $s=6$, then $H-S$ has at least two components, and every component $C$ satisfies $d_{H[C]}\ge 20$. Thus $\chi(G[C])\le |C|-20$. The components of $H-S$ use at most 12 colors in total, while $S$ uses at most six, giving $\chi(G)\le 18$. Therefore only
\begin{equation}
s\in\{29,30,31\}
\end{equation}
remain.

\subsection{\texorpdfstring{Eliminating $s=29$}{Eliminating s=29}}
Here $|X|=29$ and (43) gives at least 25 odd components. Hence every component of $H[X]$ has at most five vertices, so $\Delta(H[X])\le 4$ and $e_X\le 10$. If $e_X=0$, then $X$ is an independent 29-set in $H$, giving a proper $K_{29}$ in $G$; hence $e_X\ge 1$. By (47),
\begin{equation}
e_S\le e_X+1\le 11.
\end{equation}
Choose an edge $ab\in H[X]$. From each original triangle $T_i$ choose an edge $y_iz_i$. Since $d_S(y_i)+d_S(z_i)\le e_S+1\le 12$ and both endpoints have degree at least 26 in $H$,
\[
d_X(y_i)+d_X(z_i)\ge 40.
\]
Thus each pair $y_i,z_i$ has at least 11 common neighbors in the 29-set $X$. Choose distinct common neighbors $x_1,x_2$, both outside $\{a,b\}$. Then $x_iy_iz_i$ are two disjoint triangles.

Delete these two triangles and the edge $ab$. There remain 25 vertices on each side $S$ and $X$. Every remaining $x\in X$ had at least $26-4=22$ neighbors in $S$ and loses at most four of them, so its remaining bipartite degree is at least 18. Hall is automatic for subsets of size at most 18. If $Q$ has size $t\ge 19$ and violates Hall, then at least $r_0=26-t\in\{1,\ldots,7\}$ remaining vertices $u\in S$ have all their original $X$-neighbors among the four deleted $X$-vertices and the $25-t$ remaining vertices outside $Q$. Hence
\[
d_X(u)\le r_0+3\le 10.
\]
But (50) gives
\[
d_X(u)=d_H(u)-d_S(u)\ge 26-e_S\ge 15,
\]
a contradiction. Hall therefore provides a perfect matching. Together with the two triangles and $ab$, it partitions $H$ into 28 cliques, contrary to (3). Hence $s=29$ is impossible.

\subsection{\texorpdfstring{Eliminating $s=30$}{Eliminating s=30}}
Now $|X|=28$ and (43) gives at least 26 odd components. Consequently every component of $H[X]$ has order at most three,
\[
\Delta(H[X])\le 2,\qquad e_X\le 3,
\]
and (47) gives
\begin{equation}
e_S\le 32.
\end{equation}
Call $u\in S$ \emph{low} if $d_X(u)\le 2$. Two low vertices would each have $S$-degree at least 24 and hence would be incident with at least $24+24-1=47$ distinct $S$-edges, contradicting (51). Thus there is at most one low vertex. By Lemma~\ref{lem:minimal-barrier}, a low vertex cannot lie in $R$, so if it exists it lies in $U$.

Every triangle in $S$ has an edge whose endpoints have a common neighbor in $X$. Otherwise the three $X$-neighborhoods are pairwise disjoint, so their total size is at most 28. The sum of the three total degrees is at least 78, hence the sum of their $S$-degrees is at least 50. The number of distinct $S$-edges incident with the triangle is then at least $50-3=47$, again contradicting (51). Call such an edge \emph{good}.

First suppose there is no low vertex, or the unique low vertex $\ell$ satisfies $d_X(\ell)\ge 1$. If $\ell$ exists, choose $x_1\in N_X(\ell)$. We have $d_S(\ell)\ge 24$ and, since $\Delta(H[X])\le 2$, also $d_S(x_1)\ge 24$. The vertex $\ell$ has at least 19 neighbors in the 24-set $R$, while $x_1$ has at least 18; therefore they have at least 13 common neighbors in $R$. Choose $w$ among them. Then $\ell wx_1$ is a cross triangle. If there is no low vertex, use instead a good edge of $T_1$ and a common neighbor $x_1\in X$.

Use a good edge of the other original triangle to construct a second cross triangle. Its $X$-anchor can be chosen different from $x_1$. Indeed, if after deleting $x_1$ no edge of that triangle had a common $X$-neighbor, its three residual $X$-neighborhoods would be pairwise disjoint. Their original total $X$-degree would then be at most $27+3=30$, so the three $S$-degrees would sum to at least $78-30=48$, yielding at least 45 incident $S$-edges, contrary to (51). Thus we obtain two disjoint cross triangles with distinct $X$-anchors, containing the unique low vertex if one exists.

After deleting the two cross triangles, a $26\times 26$ bipartite graph remains. Each remaining $X$-vertex had at least 24 neighbors in $S$ before the four $S$-vertices were deleted, so its remaining degree is at least 20. If Hall fails on a set of size $t\ge 21$, put $r_0=27-t\in\{1,\ldots,6\}$. At least $r_0$ opposite vertices then have original $X$-degree at most
\[
2+(26-t)=r_0+1,
\]
where the first term accounts for the two deleted $X$-anchors. Hence each has $S$-degree at least $25-r_0$. If $r_0\ge 2$, those $r_0$ vertices are incident with at least
\[
r_0(25-r_0)-\binom{r_0}{2}\ge 45
\]
distinct $S$-edges, contradicting (51). If $r_0=1$, the bad vertex is low, but the only possible low vertex was already deleted. Thus Hall holds, producing a 28-clique partition, a contradiction.

It remains to treat a unique low vertex $\ell$ with $d_X(\ell)=0$. Then $e_X>0$, since otherwise $X\cup\{\ell\}$ is an independent 29-set in $H$. Because $|X|=28$ and there are at least 26 odd components, $H[X]$ has a unique nontrivial component $C$, of order two or three. Let $T_0$ be the original triangle containing $\ell$ and $T$ the other triangle; keep $T_0$ as an $S$-only clique.

The triangle $T$ has total $X$-degree at least 43. Indeed, at most $e_S\le 32$ distinct $S$-edges are incident with it, so the sum of its three $S$-degrees is at most 35, whereas its total degrees sum to at least 78. Therefore some edge $yz$ of $T$ has a common neighbor $x\in X\setminus C$: otherwise the neighborhoods outside $C$ would be pairwise disjoint and the total $X$-degree would be at most
\[
(28-|C|)+3|C|=28+2|C|\le 34.
\]
Let $u$ be the third vertex of $T$. Since $\ell$ is the unique low vertex, $d_X(u)\ge 3$.

Suppose an edge $ab$ of $C$ can be chosen so that
\begin{equation}
N_X(u)\not\subseteq\{a,b,x\}.
\end{equation}
Use the three cliques $T_0$, $xyz$, and $ab$. The remaining bipartite graph is $25\times25$. Every remaining $X$-vertex had at least 24 neighbors in $S$ and loses at most five, so its degree is at least 19. If Hall fails on a set of size $t\ge 20$, put $r_0=26-t$. Then at least $r_0$ opposite vertices have original $X$-degree at most $r_0+2$, accounting for the three deleted $X$-vertices. For $r_0\ge 2$, their $S$-degree is at least $24-r_0$, so they are incident with at least
\[
r_0(24-r_0)-\binom{r_0}{2}\ge 43>e_S
\]
distinct $S$-edges, impossible. For $r_0=1$, a bad vertex in $R$ would have neighbors in at most the two components represented by $C$ and the singleton $x$, contradicting Lemma~\ref{lem:minimal-barrier}; the only other possible bad vertex is $u$, and (52) leaves it an undeleted $X$-neighbor. Hence Hall holds and again yields a 28-clique partition.

The only obstruction to (52) is that $C$ has two vertices and
\begin{equation}
N_X(u)=V(C)\cup\{x\}.
\end{equation}
Indeed, if $|C|=3$, connectedness of $C$ and $d_X(u)\ge 3$ always allow an edge of $C$ omitting an $X$-neighbor of $u$ other than $x$; if $|C|=2$, its unique edge fails exactly in (53). In that last case $T\cup\{x\}$ is a $K_4$ in $H$. Use as cliques $T_0$, this $K_4$, and the edge $C$. Since $e_X>0$, $H-S$ has exactly 26 odd components, so (46) holds. Delete the singleton odd component containing $x$. The component-incidence bipartite graph still satisfies Hall for all 24 vertices of $R$, so match them to 24 distinct remaining odd components and choose an adjacent vertex in each matched component. These give 24 clique edges; one $X$-vertex remains as a singleton clique. Altogether
\[
1+1+1+24+1=28
\]
cliques cover $H$, a contradiction. Thus $s=30$ is impossible.

\subsection{\texorpdfstring{Eliminating $s=31$}{Eliminating s=31}}
Now $|X|=27$ and (43) forces all vertices of $X$ to be isolated in $H[X]$. Equation (47) and (39) give
\begin{equation}
55\le e_S\le 57.
\end{equation}
The expansion (45) holds for the 25-set $R$. Consequently, after deleting any two vertices of $X$, Hall still matches all of $R$ into the remaining 25 vertices of $X$.

Call a triangle $T\subseteq S$ \emph{splittable} if its vertices can be written $\{u,y,z\}$ and there exist distinct $x,x'\in X$ such that $xyz$ is a triangle and $ux'$ is an edge. We need one elementary estimate. If a triangle contains no vertex of $X$-degree zero and is not splittable, then it is incident with at least 48 distinct $S$-edges. If no edge of the triangle has a common $X$-neighbor, the three $X$-neighborhoods are pairwise disjoint, so their total $X$-degree is at most 27. Their total degrees are at least 78, hence their $S$-degrees sum to at least 51, giving at least $51-3=48$ incident edges. Otherwise suppose $yz$ has a common $X$-neighbor. If the opposite vertex $u$ had at least two $X$-neighbors, one could choose one distinct from the common neighbor and split the triangle. Hence $u$ has exactly one $X$-neighbor $x_0$, and non-splittability forces $N_X(y)\cap N_X(z)=\{x_0\}$. The other two triangle edges are also good through $x_0$; applying the same argument to them forces all three vertices to have exactly the single $X$-neighbor $x_0$. Their $S$-degrees are then at least 25, so the triangle is incident with at least $75-3=72$ $S$-edges.

If neither $T_1$ nor $T_2$ contains an $X$-degree-zero vertex, they cannot both be nonsplittable: two disjoint sets of at least 48 incident edges overlap in at most the nine possible edges between the triangles, which would force $e_S\ge 87$. Hence one triangle is splittable; keep the other as an $S$-only triangle. If all zero-$X$ vertices of $U$ lie in one of the two triangles, keep that triangle as the $S$-only one. The other must be splittable, for otherwise its at least 48 incident edges together with the at least 26 edges incident to one zero-$X$ vertex would imply $e_S\ge 48+26-3>57$.

Thus, unless there is one zero-$X$ vertex in each $T_i$, we have an $S$-only triangle $C$ and a splittable triangle $\{u,y,z\}$. Choose distinct $x,x'\in X$ with $xyz$ a triangle and $ux'$ an edge. After deleting $x,x'$, (45) and Hall match all 25 vertices of $R$ into the remaining 25 vertices of $X$. The cliques
\[
C,\qquad xyz,\qquad ux',\qquad\text{and the 25 matching edges}
\]
form a 28-clique partition, impossible.

Finally suppose $z_1\in T_1$ and $z_2\in T_2$ are zero-$X$ vertices. There is no third such vertex in $U$, because three vertices of $S$-degree at least 26 are incident with at least
\[
3\cdot26-\binom32=75
\]
distinct $S$-edges, contradicting (54). The vertices $z_1,z_2$ must be adjacent in $H$; otherwise $X\cup\{z_1,z_2\}$ is an independent 29-set. Their incident $S$-edges number at least $26+26-1=51$, so
\[
\e\bigl(H[S\setminus\{z_1,z_2\}]\bigr)\le 6.
\]
Every other vertex of $S$ therefore has $S$-degree at most 8, hence $X$-degree at least 18. Write the other two vertices of $T_i$ as $a_i,b_i$. Each pair $a_i,b_i$ has at least $18+18-27=9$ common neighbors in $X$, so choose distinct $x_1,x_2$ with $a_ib_ix_i$ triangles. After deleting $x_1,x_2$, (45) matches all 25 vertices of $R$ to the remaining 25 vertices of $X$. The edge $z_1z_2$, the two cross triangles, and the 25 matching edges form 28 cliques, the final contradiction.

We have exhausted (44), so the case (37) is impossible.

\begin{theorem}\label{thm:r29}
Albertson's conjecture holds for $r=29$.
\end{theorem}
\begin{proof}
The order reduction leaves only 57 and 58. Proposition~\ref{prop:no57} excludes 57. At order 58, Theorem~\ref{thm:rabern} excludes $m=841$, and the preceding complement/barrier analysis excludes $m=842$. Hence no 29-critical counterexample exists, and Lemma~\ref{lem:critical-reduction} completes the proof.
\end{proof}

\section{Completion of the proof}
\begin{proof}[Proof of Theorem~\ref{thm:main}]
For $r\le 24$ the result is Cranston's theorem. The remaining cases follow from Theorems~\ref{thm:r25},~\ref{thm:r26},~\ref{thm:r27},~\ref{thm:r28}, and~\ref{thm:r29}, respectively.
\end{proof}

\appendix
\section{Exact sampling tables}
The following tables record one sampling parameter $q$ for each order handled solely by the minimum-degree estimate. Every margin is an exact positive rational number; no decimal rounding is used.

\begin{table}[htbp]
\centering
\caption{Exact margins for $r=28$, $56\le n\le 78$.}\label{tab:r28}
\small
\begin{tabular}{r r r l}
\toprule
$n$ & $m=14n$ & $q$ & $B_q(n,m)-7098$\\
\midrule
56&784&24&$25508/759$\\
57&798&25&$14049/115$\\
58&812&25&$223237/1035$\\
59&826&26&$6822788/22425$\\
60&840&26&$596932/1495$\\
61&854&27&$7704998/15795$\\
62&868&27&$3691009/6318$\\
63&882&28&$393512/585$\\
64&896&28&$6759242/8775$\\
65&910&29&$2576/3$\\
66&924&29&$12446/13$\\
67&938&30&$254098/243$\\
68&952&30&$278330/243$\\
69&966&31&$16634698/13485$\\
70&980&31&$10795043/8091$\\
71&994&32&$15343447/10788$\\
72&1008&32&$2739681/1798$\\
73&1022&32&$8697283/5394$\\
74&1036&33&$105200417/61380$\\
75&1050&33&$7381619/4092$\\
76&1064&34&$198746387/104346$\\
77&1078&34&$75732713/37944$\\
78&1092&35&$3135587/1496$\\
\bottomrule
\end{tabular}
\end{table}

\begin{table}[htbp]
\centering
\caption{Exact margins for $r=29$, $59\le n\le 81$.}\label{tab:r29}
\small
\begin{tabular}{r r r l}
\toprule
$n$ & $m=\lceil29n/2\rceil$ & $q$ & $B_q(n,m)-8281$\\
\midrule
59&856&25&$3060428/56925$\\
60&870&25&$1079189/7590$\\
61&885&26&$4568227/17940$\\
62&899&26&$3106249/8970$\\
63&914&27&$320837/702$\\
64&928&27&$2904007/5265$\\
65&943&28&$1159634/1755$\\
66&957&28&$443359/585$\\
67&972&28&$352147/405$\\
68&986&29&$8686/9$\\
69&1001&29&$126305/117$\\
70&1015&30&$285080/243$\\
71&1030&30&$2194477/1701$\\
72&1044&31&$42841/31$\\
73&1059&31&$8097427/5394$\\
74&1073&32&$98669/62$\\
75&1088&32&$3079493/1798$\\
76&1102&33&$132691913/73656$\\
77&1117&33&$12887759/6696$\\
78&1131&34&$25449437/12648$\\
79&1146&34&$40544035/18972$\\
80&1160&35&$113386/51$\\
81&1175&35&$7030465/2992$\\
\bottomrule
\end{tabular}
\end{table}

\section*{Statement on generative AI use}
Generative AI tools, including large language models, were used substantially during the development of this work to assist with mathematical exploration, intermediate checking, and manuscript drafting. AI-generated outputs were treated as provisional and were subject to human review and independent verification where appropriate. The human authors take full responsibility for all mathematical statements, proofs, citations, computations, and conclusions presented in this work.

\end{document}